\documentclass[11pt]{article}
\usepackage{mathtools}
\usepackage{amsthm}
\usepackage{amssymb}
\usepackage{graphicx}
\usepackage[margin=1.25in]{geometry}

\newtheorem{theorem}{Theorem}

\theoremstyle{definition}

\newtheorem{corollary}[theorem]{Corollary}

\newtheorem{definition}[theorem]{Definition}

\title{Algorithmic Search in Group Theory\\{\bigskip\large Dedicated to the memory of Charles Sims}}

\author{Robert Gilman}

\author{Robert Gilman\thanks{The author thanks the Hausdorff Institute of Mathematics, the University of Newcastle and the University of Warwick for their hospitality while this paper was being written.}}

\date{\today}

\begin{document}

\maketitle

\begin{abstract} A method of random search based on Kolmogorov complexity is proposed and applied to two search problems in group theory. The method is provably effective but not practical, so the applications involve heuristic approximations. Perhaps surprisingly, these approximations seem to work. Some experimental evidence is presented.
\end{abstract}

\section{Introduction}\label{sec:intro}

One of Charlie Sims' substantial contributions to mathematics is his invention of a base and strong generating set for finite permutation groups. This invention played a crucial role in proving the existence of several sporadic finite simple groups and is the foundation of most permutation group algorithms in use today, including those used here.

The origin of this paper is a conversation some years ago between Colva Roney-Dougal and the author about the difficulty of generating random permutation groups with which to test the efficacy of various permutation group algorithms. A common way of sampling random subgroups is to choose generators at random from the ambient group. This approach fails for permutation groups because, as is well known, a random pair of permutations from the symmetric group of degree $n$ generates the symmetric or alternating group of degree $n$ with probability about $1-1/n$. We propose a different method of search which seems to do better. 

We recast the above search problem in the following general form. Given an infinite decidable subset $Y$ of a computably enumerable set $T$, search $T$ for multiple instances of $Y$. A common strategy is to decompose $T$ in some convenient way as a union of finite subsets $T=\cup T_n$, choose $n$ large, and test random elements of $T_n$ for membership in $Y$. However the search will be hard if instances of $X$ are rare, in particular if $X$ has asymptotic density 0 with respect to the decomposition of $T$:
\[lim_{n\to \infty} \frac{|X\cap T_n|}{|T_n|} = 0\]
where $|\cdot|$ denotes cardinality.

This obstacle can be avoided by choosing the decomposition $\{T_n\}$ in a special way. Let $\Sigma^*$ be the set of all words over a finite alphabet, $\Sigma$, with at least two letters; and define $C_n$ to be the finite set of all words $w$ whose Kolmogorov complexity, $C(w)$, is at most $n$.

\begin{theorem}\label{thm:main} If $X\subset \Sigma^*$ is an infinite decidable subset, then 
\[ \liminf_{n\to\infty}\frac{|X\cap C_n|}{|C_n|} > 0. \]
in other words $X$ has positive lower asymptotic density with respect to the decomposition $\Sigma^*=\cup C_n$.
\end{theorem}
Theorem~\ref{thm:main} is proved in Section~\ref{sec:proof}.
\begin{corollary}\label{thm:corollary} 
Let $\iota:\Sigma^*\to T$ be a computable bijection. If $Y\subset T$ is an infinite decidable subset, then $Y$ has positive lower asymptotic density with respect to the decomposition $T = \cup\iota(C_n)$.
\end{corollary}

Corollary~\ref{thm:corollary}, which follows from Theorem~\ref{thm:main} with $X=\iota^{-1}(Y)$, shows that searching for elements of $X$ by choosing words $w$ uniformly at random from $C_n$ and testing $\iota(w)$ for membership in $X$ succeeds with probability bounded away from $0$ for large enough $n$. We call this method algorithmic search.

Unfortunately the sets $C_n$ are intractable. If we could decide membership in $C_n$, then we could decide membership in $C_n-C_{n-1}$ and thereby compute $C(w)$, which is uncomputable~\cite[Theorem 2.3.2]{MR2494387}. Even more to the point there is no computable upper bound on the size of the largest word in $C_n$ as a function of $n$~\cite[Theorem 2.3.1]{MR2494387}. Restricting ourselves to resource bounded complexity~\cite[Chapter 7]{MR2494387} resolves the computability issues, but does not help much when it comes to practical computation. Instead in Section~\ref{sec:implementation} we use Theorem~\ref{thm:main} as a pattern for heuristic searches.

\section{Algorithmic search}\label{sec:algorithmic}

We require a few elementary results from the theory of Kolmogorov complexity. Since applications of Kolmogorov compexity to group theory are rare (we know of only~\cite{MR806660, MR2182705, MR3705849}), we sketch proofs. For a more complete introduction to the theory, the reader is referred to~\cite{MR2494387} and~\cite{MR3702040}.

\subsection{Kolmogorov complexity}\label{sec:kolmogorov}

As before fix an alphabet $\Sigma$ with $|\Sigma|\ge 2$, and let $\Sigma^*$ denote the free monoid of all words over $\Sigma$. It is customary to use $\Sigma=\{0,1\}$, but larger alphabets are convenient when working with finitely generated groups. $\Sigma^n$ is the set of words of length $n$, and $\Sigma^{\le n}$ is the set of words of length at most $n$. $\varepsilon$ is the empty word. 

The Kolmogorov complexity of a word $w\in \Sigma^*$ is the length of the shortest description of $w$. Since there relatively few  short descriptions, most words are in incompressible; that is, their complexity is not much less than their length. Incompressibility can be taken as a definition of randomness.

Descriptions of words are constructed from computer programs. Let $\mathcal L$ be a Turing complete programming language over an alphabet $\Delta$. We want to use programs in $\mathcal L$ to compute functions $\Sigma^*\to\Sigma^*$, so it is natural to assume $\Sigma\subset \Delta$. Programs in $\mathcal L$ can be coded as words over $\Sigma$ in the following way: code the letters of $\Delta$ as words of some fixed length $\ell$ over $\Sigma$, and choose $\ell$ large enough to allow an extra reserved word of length $\ell$ which signals the end of a program. Programs become $\ell$ times larger than  before, but that does not bother us. The important point is that they are words over $\Sigma$ with the reserved word as a suffix, and they can be easily decoded into their original form.

\begin{definition} A description of $w\in \Sigma^*$ is a word in $\Sigma^*$ of the form $pv$ where $p$ is a program, $v$ is a word, and $p$ with input $v$ computes $w$. The length of a shortest description of $w$ is $C(w)$, the Kolmogorov complexity of $w$.
\end{definition}
 
Our conditions imply that an arbitrary word can be written as a description $pv$ in at most one way.
 
\begin{theorem}\label{thm:complexity} The following conditions hold.
\begin{enumerate}
\item\label{item:i)} $C(w)\le |w|+c$ for some constant $c$.
\item\label{item:ii} There are at most $|\Sigma|^{n}$ words $w$ with $C(w) = n$ and at most  $|\Sigma|^{n+1}$ words with $C(w) \le n$.
\item\label{item:iii} If $f:\Sigma^*\to\Sigma^*$ is a computable function, then for any word $w$, $C(f(w))\le C(w)+ c_f$ where $c_f$ is a constant depending on $f$ but not on $w$.
\end{enumerate}
\end{theorem}
\begin{proof}
Let $p$ be a program which outputs its input and halts. For every word $w$, $pw$ is a description of $w$. Thus the first assertion holds with $c = |p|$. For the second part observe that since descriptions are words over $\Sigma$ and each word is a description in at most one way, there are at most $|\Sigma|^n$ 
descriptions of length $n$. Finally let $q$ be a program which computes $f$, and let $pv$ be a shortest description of $w$. Programs $p$ and $q$ can be combined with some overhead into a program of length at most $|p| + c_f$ which computes $f(w)$ and halts. 
\end{proof}

\subsection{Proof of Theorem~\ref{thm:main}}\label{sec:proof}

\begin{proof}
By hypothesis there is a computable bijection $f:\Sigma^* \to X$. It 
follows from Theorem~\ref{thm:complexity} that $\Sigma^n\subset C_{n+c}$. 
Likewise $f(C_{n+c})\subset X\cap C_{n+c+c_f}$ and $|C_{n+c+c_f}|\le |\Sigma|
^{n+c+c_f+1}$. Consequently
\[
\frac{|X\cap C_{n+c+c_f}|}{|C_{n+c+c_f}|} \ge \frac{|\Sigma|^n}{|\Sigma|^{n+c+c_f+1}} = \frac{1}{|\Sigma|^{c+c_f+1}}
\]
\end{proof}

\section{Implementation}\label{sec:implementation}

The algorithmic search method described in Section~\ref{sec:algorithmic} is impractical because, as we know from Section~\ref{sec:intro}, the sets $C_n$ are intractable. This section is devoted to two heuristic variations, which we also call algorithmic search. Instead of choosing random words in $C_n$, we choose random short descriptions, and to facilitate this choice we restrict the programs allowed in descriptions. 

Our heuristic variations are preliminary; ease of programming was a primary consideration. Nevertheless the results seem encouraging. Computations were done with the Magma system~\cite{MR1484478}. Figure~\ref{fig:permutations} required 40 hours of CPU time on a decent laptop.

\subsection{Finitely generated groups}\label{sec:words}
Let $\Sigma^*\to G$ be a choice of semigroup generators for the infinite group 
$G$, and suppose we wish to choose random elements of $G$. In the case of finite groups the product replacement algorithm effectively approximates the uniform distribution on the group~\cite{MR1829489}, but there is no uniform distribution on an infinite group. In practice it seems reasonable 
that $\overline w$, the image in $G$ of $w\in\Sigma^*$, is close to 
random if for some large $n$, $w$ is chosen at uniformly at random from $\Sigma^{\le n}$. But then sets $X \subset \Sigma^*$ of asymptotic density $0$ with respect to the decomposition $\Sigma^*=\cup \Sigma^{\le n}$ are invisible. In particular $\overline w$ is never equal to $1$ in $G$~\cite[Theorem 5.7]{MR1981427}. The disadvantage for, say, debugging and testing algorithms is obvious. Algorithmic search seems to do better.

\begin{figure}[!hb]
\centerline{\includegraphics[scale=.6]{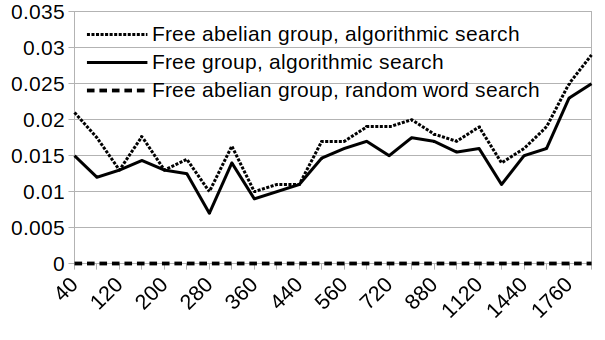}}
\caption{Observed frequency of words of various lengths defining the identity in the free group of rank 2 and the free abelian group of rank 2. Every data point for algorithmic search was computed from 1000 or more random descriptions, each of which was evaluated in both the free group and the free abelian group.  Data points for random word search were computed from over 1000 choices of random words in the generators. Random word search produced no words defining the identity.\label{fig:words}}
\end{figure}

For the algorithmic search used to produce Figure~\ref{fig:words}, descriptions have the form $pv$ as before, but $p$ is defined to be a sequence of monoid homomorphisms. Each homomorphism is given by listing the images of the letters in $\Sigma$ under that homomorphism. The word described by $pv$ is the image of $v$ under the composition of the homomorphisms. Our semigroup generators are $\Sigma=\{a,A,b,B\}$ where were are writing $A, B$ in place of the customary formal inverses $a^{-1}, b^{-1}$. For example the the description
\[ pv = \underbrace{ab,bB,aB,AA,}_{f}\underbrace{bb,BA, BB, AB,}_{g} abAA\]
describes the word $w=f\circ g(abAA) = aBaBaBaBaBbBaBaB$. Here we are not adhering strictly to the format from Section~\ref{sec:kolmogorov}.

$C'_{d,c,M}$ denotes the set of all descriptions $pv$ with $d$ homomorphisms, each specified by a tuple of words of length $c$, and with $|v|\le M$. Algorithmic search is performed by choosing random descriptions from $C'_{d,c,M}$ for various choices of the parameters, computing the words described, and testing them to see if they define the identity in the two groups from Figure~\ref{fig:words}.

\subsection{Permutation groups}\label{sec:permutations}

As mentioned above, pairs of permutations chosen at random from the symmetric group $S_n$ are unlikely to generate anything but $S_n$ or $A_n$. More precisely we have the following theorem.

\begin{theorem}[\cite{MR3391478}]\label{thm:Colva}
Two random permutations in $S_n$ generate a subgroup other than $S_n$ or $A_n$ with probability at most $\frac{1}{n} +\frac{8.8}{n^2}$. 
\end{theorem}

In order to apply algorithmic search to the permutation group search problem from Section~\ref{sec:intro} we reformulate that problem slightly.  $S_\omega$ is the group of all permutations of $N=\{1,2,\ldots\}$ with finite support. $S_n$ acts on $\{1,\ldots, n\}$ in the usual way, and fixes all other elements of $N$. $\Sigma=\{0,1\}$ and $\iota:\Sigma^*\to S_\omega \times S_\omega$ is a computable bijection as in Corollary~\ref{thm:corollary}.  $Y\subset  S_\omega \times S_\omega$ is the collections of all of pairs of permutations which do not generate any $S_n$ listed above or its alternating group.

In the original formulation of the search problem, a pair of permutations from $S_n$ is ruled out if it generates $S_n$ or $A_n$. In the revision a pair from $S_\omega$ is ruled out if it generates any $S_n$ listed above or its alternating group. The bound of Theorem~\ref{thm:Colva} still applies. 

\begin{figure}[!ht]
\includegraphics[scale=.52]{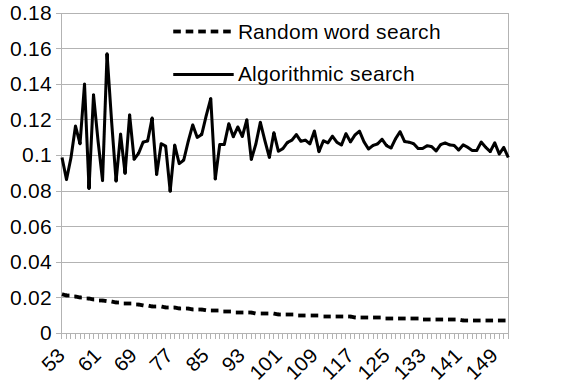}
\includegraphics[scale=.52]{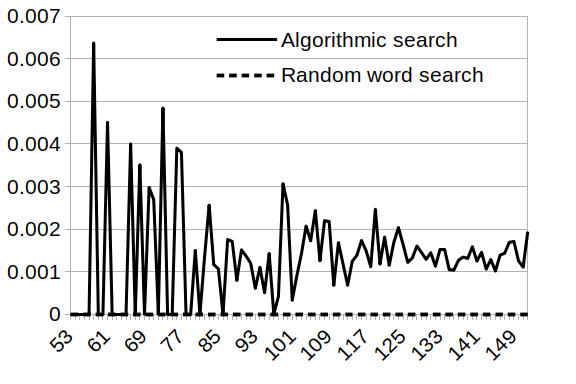}
\caption{Observed frequencies of 2-generator permutation groups which are not $S_n$ or $A_n$ (left) and which are solvable (right) for various permutation degrees. The dashed line on the left is the upper bound from Theorem~\ref{thm:Colva}. The dashed line on the right shows that random words in the generators never generated a solvable permutation group.\label{fig:permutations}}
\end{figure}

The map $\iota:\Sigma^*\to S_\omega \times S_\omega$ is a composition of  bijections $\Sigma^*\to N \to N\times N \to S_\omega \times S_\omega$.
$\Sigma^*\to N$ is the correspondence 
\[
\begin{array}{r|ccccccccc}
\Sigma^*&\varepsilon,  &0 &1 &00 &01 &10 &11 &000 &\cdots\\\hline
N         &1 &2  &3  &4  &5  &6  &7 &8  &\cdots
\end{array}
\]
while $N\to N\times N$ is a well known way of enumerating $N\times N$
\[
\begin{array}{r|ccccccccc}
N         &1 &2  &3  &4  &5  &6  &7 &8  &\cdots\\\hline
N\times N &(1,1) &(1,2)  &(2,1)  &(1,3)  &(2,2)  &(3,1)  & (1,4) &(2,3)  &\cdots
\end{array}
\]
and $N\times N \to S_\omega\times S_\omega$, is constructed from a standard enumeration of all permutations~\cite[Section 7.2.1.2]{MR3444818}.

Algorithmic search in this case resembles that of the preceding section except that instead of descriptions based on monoid homomorphisms from $\Sigma^*$ to $\Sigma^*$ we employ descriptions based on polynomial functions from $N$ to $N$. The polynomials have non-negative integer coefficients. For example the description 
\[\underbrace{8,2,3,1}_{p};\underbrace{6,7,4,2}_{q};15\]
with 2 degree 3 polynomials describes the integer
\[(8x^3+2x^2+3x+1) \circ (6x^3+7x^2+4x+2)(15) =  83879080636024\]
which gets mapped to the pair of permutations
\[ (1, 7, 8, 11, 10, 6, 2, 3, 4, 9, 5)\quad (1, 2, 6, 5, 7, 3, 4, 10, 11, 9, 8).\]

Figure~\ref{fig:permutations} shows results obtained by selecting 1,000,000 descriptions uniformly at random from the set of all descriptions with 7 degree 2 polynomials, $ax^2+bx + c$, satisfying $1\le a \le 20$, $0\le b,c \le 20$, and with $|v|\le 1000$. It appears that $S_n$ and $An$ are avoided about 10\% of the time and solvable permutation groups are obtained about $.1\%$ of the time. Whether or not results like these are useful for the permutation group search problem is not clear. In any case the method can probably be refined.

\bibliography{Sims}{}
\bibliographystyle{plain}

\end{document}